\documentclass[orivec,envcountsame]{llncs}

\usepackage{amsmath,amssymb}
\usepackage{enumerate}

\newcommand{\starshaped}{\text{Star}}
\newcommand{\discardcolor}{\text{DiscardColor}}
\newcommand{\discardnode}{\text{DiscardNode}}
\newcommand{\renumnode}{\text{RenumNode}}
\newcommand{\renumcolor}{\text{RenumColor}}

\def\vecp{{\vec p\thinspace}}
\def\vecpprime{{\vec p^{\thinspace\prime}}}

\def\limplies{\rightarrow}

\def\bigand{\bigwedge}
\def\bigor{\bigvee}

\newcommand{\bracket}[1]{[#1]}
\newcommand{\kneser}[2]{\mathrm{Kneser}^{#1}_{#2}}

\title{Short Proofs of the \\ Kneser-Lov\'asz
Coloring Principle}
\author{
James Aisenberg \inst{1} \fnmsep \thanks{Supported
in part by NSF grants DMS-1101228 and CCF-1213151.}
\and
Maria Luisa Bonet \inst{2} \fnmsep \thanks{Supported in part by grant TIN2013-48031-C4-1.}
\and
Sam Buss \inst{1} \fnmsep \thanks{Supported
in part by NSF grants DMS-1101228 and CCF-1213151,
and Simons Foundation award 306202.}
\and
Adrian Cr\~{a}ciun\inst{3} \fnmsep \thanks{Supported in part by IDEI grant PN-II-ID-PCE-2011-3-0981``Structure and computational difficulty in combinatorial optimization: an interdisciplinary approach''.}
\and
Gabriel Istrate \inst{3} \fnmsep ${}^\dagger$}

\institute{
Department of Mathematics,
University of California, San Diego,
La Jolla, CA 92093-0112, USA
\email{jaisenberg@math.ucsd.edu,~sbuss@ucsd.edu}
\and
Computer Science Department,
Universidad Polit{\'e}cnica de Catalu{\~n}a,
Barcelona, Spain
\email{bonet@cs.upc.edu}
\and
West University of Timi\c{s}oara,
and the e-Austria Research Institute,
Timi\c{s}oara, RO-300223, Romania
\email{acraciun@ieat.ro,~gabrielistrate@acm.org}
}

\begin{document}
\maketitle
\begin{abstract}
We prove that the propositional translations
of the Kneser-Lov\'asz theorem have polynomial size
extended Frege proofs and quasi-polynomial size
Frege proofs.  We present a new
counting-based combinatorial proof of the Kneser-Lov\'asz
theorem that avoids the topological arguments
of prior proofs for all but finitely many cases for each~$k$.
We introduce a miniaturization
of the octahedral Tucker lemma, called the {\em truncated Tucker lemma}:
it is open whether its propositional translations have
(quasi-)polynomial size Frege or
extended Frege proofs.
\end{abstract}

\section{Introduction}

This paper discusses proofs of Lov\'asz's theorem about
the chromatic number of Kneser graphs, and the proof
complexity of propositional translations of the
Kneser-Lov\'asz theorem.  We give a new proof of the
Kneser-Lov\'asz theorem that uses a simple counting
argument instead of the topological arguments used
in prior proofs, for all but finitely many cases.
Our arguments can be formalized
in propositional logic to give
polynomial size extended Frege proofs and
quasi-polynomial size Frege proofs.

Frege systems are sound and complete proof systems for propositional logic
with a finite schema of axioms and inference rules.  The typical
example is a ``textbook style'' propositional proof system using
{\em modus ponens} as its only rule of inference, and all Frege
systems are polynomially equivalent to this system~\cite{CookReckhow:proofs}.
Extended Frege systems are Frege systems augmented with the extension
rule, which allows
variables to abbreviate complex formulas.
The {\em size} of a Frege or extended Frege proof is measured by counting the
number of symbols in the proof \cite{CookReckhow:proofs}.
Frege proofs are able to reason using Boolean formulas;
whereas extended Frege proofs can reason using Boolean circuits
(see~\cite{Jerabek:dualweakphp}).
Boolean formulas are conjectured to require exponential size to simulate
Boolean circuits; there is no known direct connection, but by analogy,
it is generally conjectured that there
is an exponential separation between
the sizes of Frege proofs and extended Frege proofs.
This is one of the important open questions in
proof complexity; for more on proof complexity see
e.g.~\cite{BBP:hardFrege,%
Buss:marktoberdorf2,%
Buss:lfcsSurveyAPAL,%
CookReckhow:proofs,%
Krajicek:book,%
Segerlind:ProofComplexity}.

As discussed by
Bonet, Buss and Pitassi~\cite{BBP:hardFrege} and more
recently by \cite{ABB:FranklsThm,Buss:QuasiPolyPHP},
we have hardly any examples of combinatorial
tautologies, apart from consistency statements,
that are conjectured to exponentially separate Frege and
extended Frege proof size.  These prior works discussed
a number of combinatorial
principles, including the pigeonhole principle and Frankl's theorem.
Istrate and Cr{\~a}ciun~\cite{IstrateCraciun:KneserLovasz}
recently proposed the Kneser-Lov\'asz principle as a
candidate for
exponentially separating Frege and
extended Frege proof size.
In this paper we give quasi-polynomial size Frege proofs
of the propositional translations of the Kneser-Lov\'asz
theorem for all fixed~$k$.
Thus they do not provide an exponential separation
of Frege and extended Frege proof size.

Our proof is also interesting because it gives a new method
of proving the Kneser-Lov\'asz theorem.
Prior proofs use
(at least implicitly) a topological fixed-point lemma.  The
most combinatorial proof is by Matou\v sek~\cite{Matousek:ComboProofsKneser} and is
inspired by the octahedral Tucker lemma; see also Ziegler~\cite{Ziegler:GeneralKneser}.
Our new proofs mostly avoid topological arguments
and use a counting argument instead.
These counting arguments can be formalized with Frege proofs.
Indeed, one of the important strengths of Frege proofs is that they can
reason about integer arithmetic.  These techniques
originated in polynomial size Frege proofs
of the pigeonhole principle~\cite{Buss:PHP} which used
carry-save-addition representations for vector
addition and multiplication in order to express and prove properties about
integer operations in polynomial size.
For the Kneser-Lov\'asz theorem,
the counting arguments
reduce the general case to
``small'' instances
of size $n \le 2 k^4$.  For fixed~$k$, there
are only finitely many small instances, and they can be verified
by exhaustive enumeration.  As we shall see, this leads to polynomial
size extended Frege proofs, and quasi-polynomial size Frege proofs,
for the Kneser-Lov\'asz principles.

It is surprising that the topological arguments can be largely
eliminated from the proof of the Kneser-Lov\'asz theorem.  The only
remaining use of topological arguments is to establish the
``small instances''.   It would be
interesting to give an additional argument that avoids having
to prove the small instances separately.
One possibility for this would be to adapt the proof
based on the octahedral Tucker lemma to quasi-polynomial size
Frege proofs. The first difficulty with this is that the octahedral Tucker
lemma has exponentially large propositional translations.
To circumvent this,
we present a miniaturized
version of the octahedral Tucker lemma called the {\em truncated
Tucker lemma}.  The truncated Tucker lemma has polynomial
size propositional translations.  We prove that
the Kneser-Lov\'asz tautologies have polynomial size
constant depth Frege proofs if the propositional
formulas for the
truncated Tucker lemma are given as additional
hypotheses.  However, it remains open whether these truncated
Tucker lemma principles have (quasi-)polynomial size
Frege or extended Frege proofs.



The $(n,k)$-Kneser graph is defined to be
the undirected graph whose vertices are the $k$-subsets of $\{1, \dots, n\}$;
there is an edge between two vertices iff those vertices have empty intersection.
The Kneser-Lov\'asz theorem states that Kneser graphs have a large chromatic
number:

\begin{theorem}[Lov\'asz \cite{Lovasz:KneserTheorem}]\label{thm:KneserLovasz}
Let $n \ge 2k > 1$.
The $(n,k$)-Kneser graph has no coloring with $n-2k+1$ colors.
\end{theorem}

It is well-known that the $(n,k)$-Kneser graph has a coloring with
$n-2k+2$ colors
(see e.g.\ the appendix to the arXiv version of this paper),
so the bound $n-2k+1$
is optimal. For $k=1$, the Kneser-Lov\'asz theorem is just the
pigeonhole principle.

Istrate and Cr{\~a}ciun~\cite{IstrateCraciun:KneserLovasz} noted that,
for fixed values of~$k$, the
propositional translations of the Kneser-Lov\'asz theorem have polynomial size in~$n$. They presented arguments that can be formalized by polynomial size Frege proofs
for $k=2$, and by polynomial size extended Frege proofs for $k=3$.
This left open
the possibility that the $k=3$ case could exponentially separate the Frege and
extended Frege systems. It was also left open whether
the $k>3$ case of the Kneser-Lov\'asz theorem gave tautologies
that require exponential size extended Frege proofs.
As discussed above, the present paper refutes these possibilities.
Theorems \ref{thm:MainThmEF} and~\ref{thm:MainThmFrege} summarize
these results.

Let $\bracket{n}$ be the set $\{1, \dots, n\}$;
members of~$\bracket n$ are called {\em nodes}.  We identify $\binom{n}{k}$
with the set of $k$-subsets of $\bracket{n}$,
the {\em vertices} of the $(n,k)$-Kneser graph.

\begin{definition}
An {\em $m$-coloring} of the $(n,k)$-Kneser graph
is a map~$c$ from $\binom{n}{k}$ to~$\bracket{m}$, such
that for $S,T \in \binom{n}{k}$, if $S\cap T = \emptyset$, then $c(S)\not=c(T)$.
If $\ell\in \bracket{m}$,
then the {\em color class} $P_\ell$ is the
set of vertices assigned the color~$\ell$ by~$c$.
\end{definition}

The formulas $\kneser{n}{k}$ are the natural
propositional translations
of the statement that there is no $(n-2k+1)$-coloring
of the $(n,k)$-Kneser graph:

\begin{definition} \label{def:kneser}
Let $n \ge 2k >1$, and $m=n-2k+1$.
For $S \in \binom{n}{k}$ and
$i \in \bracket{m}$, the propositional variable~$p_{S,i}$ has
the intended meaning that vertex~$S$ of the Kneser
graph is assigned the color~$i$.  The formula $\kneser{n}{k}$ is
\[
\bigand_{S \in \binom{n}{k}} \, \bigor_{i \in \bracket{m}} p_{S,i}
   ~\limplies~
   \bigor_{\substack{S,T \in \binom{n}{k} \\ S \cap T = \emptyset}} \,
     \bigor_{i \in \bracket{m}}
       \left ( p_{S,i} \wedge p_{T,i} \right ).
\]
\end{definition}

\begin{theorem}\label{thm:MainThmEF}
For fixed parameter $k \ge 1$, the propositional translations
$\kneser n k$ of the
Kneser-Lov\'asz theorem have polynomial size extended Frege proofs.
\end{theorem}

\begin{theorem}\label{thm:MainThmFrege}
For fixed parameter $k \ge 1$, the propositional translations
$\kneser n k$ of the
Kneser-Lov\'asz theorem have quasi-polynomial size Frege proofs.
\end{theorem}

When both $k$ and~$n$ are allowed to vary, it is
open  whether the $\kneser n k$ tautologies have quasi-polynomial
size (extended) Frege proofs, or equivalently,
have proofs with size quasi-polynomially bounded in terms of~$n^k$.

\section{Mathematical Arguments}\label{sec:MathProof}

Section~\ref{sec:MathProofEF} gives the new proof of the
Kneser-Lov\'asz theorem; this is later shown to be formalizable
with polynomial size extended Frege proofs.
Section~\ref{sec:MathProofFrege} gives a slightly more complicated
but more efficient
proof, later shown to be formalizable
with quasi-polynomial size Frege proofs.
The next definition and lemma are crucial for
Sects.\ \ref{sec:MathProofEF} and~\ref{sec:MathProofFrege}.

Any two vertices
in a color class~$P_\ell$ have non-empty intersection.  One way this can
happen is for the color class to be ``star-shaped'':

\begin{definition}\label{def:StarShaped}
A color class $P_\ell$ is {\em star-shaped} if
$\bigcap P_\ell$ is non-empty.
If $P_\ell$ is star-shaped, then any
$i\in \bigcap P_\ell$ is called a {\em central element} of $P_\ell$.
\end{definition}

The next lemma bounds the size of color classes
that are not star-shaped. It will be used in our proof of the Kneser-Lov\'asz theorem
to establish the existence of star-shaped color classes.
The idea is that non-star-shaped color classes are too small
to cover all $\binom n k$ vertices.

\begin{lemma}\label{lem:NonStarShaped}
Let $c$ be a coloring of $\binom{n}{k}$.
If $P_\ell$ is not star-shaped, then
\[
|P_\ell| ~ \le ~k^2 \binom{n-2}{k-2}.
\]
\end{lemma}

\begin{proof}
Suppose $P_\ell$ is not star-shaped.
If $P_\ell$ is empty, the claim is trivial.  So suppose $P_\ell\not=\emptyset$,
and let $S_0=\{a_1, \dots, a_k\}$ be some element of~$P_\ell$.
Since $P_\ell$ is not star-shaped, there must be sets $S_1, \dots, S_k \in P_\ell$
with $a_i \notin S_i$ for $i=1, \dots, k$.

To specify an arbitrary element~$S$ of~$P_\ell$, we do the following.
Since $S$ and $S_0$ have the same color,
$S \cap\penalty10000 S_0$ is non-empty.
We first specify some $a_i \in S \cap\penalty10000 S_0$.
Likewise, $S \cap S_i$ is non-empty; we second specify some $a_j \in S  \cap S_i$.
By construction, $a_i \neq a_j$,
so $S$ is fully specified by the $k$ possible values for~$a_i$, the $k$ possible
values for~$a_j$,
and the $\binom{n-2}{k-2}$ possible values for the remaining members of~$S$.
Therefore,
$|P_\ell| \le k^2 \binom{n-2}{k-2}$.
\qed
\end{proof}

\subsection{Argument for Extended Frege Proofs}\label{sec:MathProofEF}
\label{sec:KneserLovaszExtendedFrege}

Let $k>1$ be fixed. We prove the Kneser-Lov\'asz theorem by
induction on~$n$. The base cases for the induction
are $n=2k, \dots, N(k)$ where $N(k)$ is
the constant depending on~$k$ specified in Lemma~\ref{lem:SomeStarShaped}.
We shall show that $N(k)$ is no greater than~$k^4$.
Since $k$ is fixed, there are only finitely many base cases.
Since the Kneser-Lov\'asz theorem
is true, these base cases can all be proved by
a fixed Frege proof of finite size (depending on~$k$).
Therefore, in our proof below, we only show the induction step.

\begin{lemma}\label{lem:SomeStarShaped}
Fix $k>1$. There is an $N(k)$ so that, for $n>N(k)$, any
$(n-2k+1)$-coloring of~$\binom{n}{k}$ has at least one star-shaped color class.
\end{lemma}

\begin{proof}
Suppose that a coloring~$c$ has no star-shaped color class.
Since there are $n-2k+1$ many color classes,
Lemma~\ref{lem:NonStarShaped} implies that
\begin{equation}\label{eq:KneserExtendedFrege}
(n-2k+1) \cdot k^2 \binom{n-2}{k-2} ~\geq~ \binom{n}{k}.
\end{equation}
For fixed~$k$, the left-hand side of~(\ref{eq:KneserExtendedFrege})
is $\Theta(n^{k-1})$ and the
right-hand side is $\Theta(n^k)$.
Thus, there exists an~$N(k)$
such that
(\ref{eq:KneserExtendedFrege}) fails for all $n>N(k)$.
Hence for $n>N(k)$, there must be
at least one star-shaped color class.
\qed
\end{proof}

To obtain an upper bound on the value of~$N(k)$, note that
(\ref{eq:KneserExtendedFrege}) is
equivalent to
\begin{equation}\label{eq:KnesereF2}
(n-2k+1) k^3(k-1) ~\ge~ n(n-1).
\end{equation}
Since $2k-1\ge 1$, (\ref{eq:KnesereF2}) implies that
$(n-1)k^4 > n(n-1)$ and thus that $n < k^4$.
Thus, (\ref{eq:KneserExtendedFrege}) will be
false if $n\ge k^4$; so
$N(k) < k^4$.

We are now ready to give our first proof of the
Kneser-Lov\'asz theorem.

\begin{proof}[of Theorem~\ref{thm:KneserLovasz}, except for base cases]
Fix $k>1$. By Lemma~\ref{lem:SomeStarShaped}, there is some~$N(k)$ such that
for $n>N(k)$, any $(n-2k+1)$-coloring~$c$ of~$\binom{n}{k}$ has a star-shaped
color class.  As discussed above, the cases of $n \le N(k)$ cases are handled by
exhaustive search and the truth of the Kneser-Lov\'asz theorem.
For $n>N(k)$, we prove the claim by infinite descent.
In other words, we show that if $c$ is an $(n-2k+1)$-coloring
of~$\binom{n}{k}$, then there is some~$c^\prime$ which is an
$((n-1) - 2k+1)$-coloring of~$\binom{n-1}{k}$.

By Lemma~\ref{lem:SomeStarShaped}, the coloring~$c$ has some star-shaped
color class~$P_\ell$ with central element~$i$.
Without loss of generality, $i=n$ and $\ell=n-2k+1$.
Let
\[
c^\prime ~=~ c \restriction {\textstyle \binom{n-1}{k}}
\]
be the restriction of~$c$ to the domain $\binom{n-1}{k}$.
This discards the central element~$n$ of~$P_\ell$, and
thus all vertices with color~$\ell$.  Therefore,
$c^\prime$ is an $((n-1)-2k+1)$-coloring of $\binom{n-1}{k}$.
This completes the proof.
\qed
\end{proof}

\subsection{Argument for Frege Proofs}\label{sec:MathProofFrege}
\label{sec:KneserLovaszFrege}

We now give a second proof of the Kneser-Lov\'asz theorem.  The proof
above required $n-N(k)$ rounds of infinite descent to transform
a Kneser graph on $n$ nodes to one on $N(k)$ nodes.  Our second proof
replaces this with only $O(\log n)$ many rounds, and this efficiency will be key
for formalizing this proof with quasi-polynomial size Frege proofs
in Sect.~\ref{sec:FormalF}.

We refine Lemma~\ref{lem:SomeStarShaped} to show that for~$n$
sufficiently large, there are many (i.e., a constant fraction)
star-shaped color classes.
The idea is to combine the upper bound of Lemma~\ref{lem:NonStarShaped}
on the size of non-star-shaped color classes with the trivial upper bound of
$\binom{n-1}{k-1}$ on the size of star-shaped color classes.

\begin{lemma}\label{lem:ManyStarShaped}
Fix $k>1$ and $0<\beta<1$. Then there exists an $N(k,\beta)$ such
that for $n> N(k,\beta)$, if $c$
is an $(n-2k+1)$-coloring of $\binom{n}{k}$, then $c$ has at
least $\frac{n}{k}\beta$ many star-shaped color classes.
\end{lemma}

\begin{proof}
The value of $N(k,\beta)$ can be set equal to $\frac{k^3(k-\beta)}{1-\beta}$.
Let $n> \frac{k^3(k-\beta)}{1-\beta}$,
 and suppose $c$ is an
$(n-\penalty10000 2k+\penalty10000 1)$-coloring of $\binom n k$.
Let $\alpha$ be the number of star-shaped color classes of~$c$.
It is clear that an upper bound on the size of each star-shaped color class is
$\binom{n-1}{k-1}$.
There are $n-\alpha-2k+1$ many non-star-shaped classes,
and Lemma~\ref{lem:NonStarShaped} bounds their size by
$k^2\binom{n-2}{k-2}$.
This implies that
\begin{equation}\label{eq:iSS1}
\binom{n-1}{k-1}\alpha+k^2\binom{n-2}{k-2}(n-\alpha-2k+1)~\ge~\binom{n}{k}.
\end{equation}
Assume for a contradiction that $\alpha<\frac{n}{k} \beta$.
Since $n >\frac{k^3(k-\beta)}{1-\beta}$, $0<\beta<\penalty10000 1$, and $k\ge 2$,
we have $n-1 > k^3(k-1) > k^2(k-1)$.
Therefore, $\binom{n-1}{k-1} > k^2\binom{n-2}{k-2}$, and
if $\alpha$ is replaced by the larger value $\frac n k \beta$,
the left hand side of~(\ref{eq:iSS1}) increases.
Thus,
\[
\binom{n-1}{k-1}\frac{n}{k}\beta
  + k^2\binom{n-2}{k-2} \Bigl (n-\frac{n}{k}\beta-2k+1 \Bigr )
  ~>~ \binom{n}{k}.
\]
Since $\binom{n-1}{k-1}\frac{n}{k}=\binom{n}{k}$ and
$n-\frac{n}{k}\beta-2k+1=\frac{k-\beta}{k}n-2k+1$,
\begin{equation}\label{eq:April}
k^2\binom{n-2}{k-2}\Bigl(\frac{k-\beta}{k}n-2k+1\Bigr)
   ~>~ (1-\beta)\binom{n}{k}.
\end{equation}
We have $\frac{k-\beta}{k}(n-1) > \frac{k-\beta}{k}n-2k+1$.
Therefore, (\ref{eq:April}) gives
\[
k^3(k-1)\frac{k-\beta}{k}(n-1)~>~(1-\beta)n(n-1).
\]
Dividing by $n-1$ gives $k^3(k-\beta) ~>~ (1-\beta)n$,
contradicting $n>\frac{k^3(k-\beta)}{1-\beta}$.
\qed
\end{proof}

We now give our second proof of the Kneser-Lov\'asz theorem.

\begin{proof}[of Theorem~\ref{thm:KneserLovasz}, except for base cases]
Fix $k>1$. By Lemma~\ref{lem:ManyStarShaped} with $\beta = 1/2$,
if $n>N(k,1/2)$ and $c$ is an $(n-2k+1)$-coloring of
$\binom{n}{k}$, then $c$ has at least $n/2k$ many star-shaped color classes.
We prove the Kneser-Lov\'asz theorem by induction on~$n$.
The base cases are for $2k \le n \le N(k,1/2)$, and there are only finitely
of these, so they can be exhaustively proven.
For $n>N(k,1/2)$, we structure the induction proof as an infinite descent.
In other words, we show that if $c$ is an $(n-2k+1)$-coloring of~$\binom{n}{k}$,
then there is some~$c^\prime$ that is an
$((n-\frac{n}{2k}) - 2k+1)$-coloring of $\binom{n-\frac{n}{2k}}{k}$.
For simplicity of notation, we assume $\frac n{2k}$
is an integer. If this is not the case, we really mean
to round up to the nearest integer $\lceil\frac n {2k} \rceil$.

By permuting the color classes and the nodes,
we can assume w.l.o.g.\ that the $\frac n {2k}$
color classes~$P_\ell$ for $\ell= n-\frac n{2k} -\penalty10000 2k+\penalty10000 2, \ldots, n-2k+1$ are
star-shaped, and each such $P_\ell$ has central element~$\ell+2k-1$.
That is, the
last $\frac n{2k}$ many color classes are star-shaped and their central elements
are the last $\frac n{2k}$ nodes in~$\bracket{n}$.
(It is possible that some star-shaped color classes share central nodes;
in this case, additional nodes can be discarded so that $n/2k$ are
discarded in all.)

Define $c^\prime$ to be the coloring of $\binom{n-n/2k}{k}$ which assigns the
same colors as~$c$.
The map~$c^\prime$ is a
$(\frac{2k-1}{2k}n - 2k+1)$-coloring of
$\displaystyle \binom{\frac{2k-1}{2k}n}{k}$,
since $n-\frac{n}{2k} = \frac{2k-1}{2k}n$.
This completes the proof of the induction step.
\qed
\end{proof}

When formalizing the above argument with quasi-polynomial size Frege proofs,
it will be important to know how many iterations of the procedure
are required to reach the base cases, so let us calculate this.

After $s$ iterations of this procedure, we have a
$((\frac{2k-1}{2k})^s n - 2k+1)$-coloring of
$\displaystyle \binom{(\frac{2k-1}{2k})^s n}{k}$.
We pick $s$ large enough so that $(\frac{2k-1}{2k})^s n$
is less than $N(k,1/2)$. In other words, since $k$ is constant,
\[
s  ~=~ \log_{\frac{2k}{2k-1}}\Bigl(\frac{n}{k^3(2k-1)}\Bigr) ~=~ O( \log n )
\]
will suffice, and
only $O(\log n)$ many rounds of the procedure are required.

We do not know if the bound in Lemma~\ref{lem:ManyStarShaped}
is optimal or close to optimal.  An appendix in the arXiv version
of this paper
will discuss the best examples we know of colorings with large numbers
of non-star-shaped color classes.

\section{Formalization in Propositional Logic}
\subsection{Polynomial Size Extended Frege Proofs}
\label{sec:FormalEF}

We sketch the formalization
of the argument in Sect.~\ref{sec:KneserLovaszExtendedFrege}
as a polynomial size extended Frege proof, establishing
Theorem~\ref{thm:MainThmEF}.
We concentrate on showing how to express concepts such
as ``star-shaped color class'' with polynomial size propositional
formulas.  For space reasons, we omit the straightforward details
of how (extended) Frege proofs can prove properties of
these concepts.

Fix values for $k$ and~$n$ with $n>N(k)$.  We describe an extended Frege
proof of $\kneser n k$.
We have variables $p_{S,j}$ (recall Definition~\ref{def:kneser}), collectively denoted
just $\vecp$.   The proof assumes $\kneser n k(\vecp)$
is false, and proceeds by contradiction.
The main step is to define new variables $\vecp^\prime$
and prove that $\kneser{n-1} k (\vecpprime)$ fails.
This will be repeated until reaching a Kneser graph over only
$N(k)$ nodes.

For this, let $\starshaped(i, \ell)$ be a formula
that is true when $i \in \bracket{n}$ is a central element
of the color class~$P_\ell$; namely,
\[
\starshaped(i,\ell) ~:=~
   \bigand_{S\in \binom n k , \, i\notin S} \neg p_{S,\ell}.
\]
We use
$\starshaped(\ell) := \bigor_i \starshaped(i,\ell)$
to express that $P_\ell$ is star-shaped.

The extended Frege proof defines the instance of
the Kneser-Lovasz principle $\kneser{n-1}k$ by discarding one node and one
color.  The first star-shaped color class~$P_\ell$ is discarded; accordingly,
we let
\[
\discardcolor(\ell) ~:=~ \starshaped(\ell) \land \bigand_{\ell^\prime<\ell} \neg \starshaped(\ell^\prime).
\]
The node to be discarded is the least central element of the discarded~$P_\ell$:
\[
\discardnode(i) := \bigor_{\ell} \Bigl[ \discardcolor(\ell) \land
    \starshaped(i,\ell) \land \bigand_{i^\prime<i} \neg\starshaped(i^\prime,\ell) \Bigr].
\]
After discarding the node~$i$ and color class~$P_\ell$, the remaining nodes
and colors are renumbered to the ranges $[n-1]$ and~$[n-2k]$, respectively.
In particular, the ``new'' color~$j$ (in the instance of $\kneser {n-1} k$)
corresponds
to the ``old'' color $j^{-\ell}$ (in the instance of $\kneser n k$) where
\[
j^{-\ell} = \begin{cases}
j    & \text{if $j<\ell$} \\
j +1 & \text{if $j \ge \ell$.}
\end{cases}
\]
And, if $S=\{i_1, \dots, i_k \} \in \binom{n-1}{k}$
is a ``new'' vertex (for the $\kneser {n-1} k$ instance), then
it corresponds to the ``old'' vertex $S^{-i}\in \binom n k$ (for the instance
of $\kneser n k$),
where $S^{-i} = \{i_1^\prime, i_2^\prime, \dots, i_k^\prime \}$ with
\[
i_t^\prime = \begin{cases}
i_t    & \text{if $i_t <i$} \\
i_t +1 & \text{if $i_t \ge i$.}
\end{cases}
\]
For each $S \in \binom{n-1}{k}$ and $j \in \bracket{n-1}$,
the extended Frege proof uses the extension rule to introduce a new variable
$p^\prime_{S,j}$ defined as follows
\[
p^\prime_{S,j} \equiv \bigor_{i, \ell } \left ( \discardnode(i)\wedge \discardcolor(\ell)
  \wedge p_{S^{-i},j^{-\ell}}    \right ).
\]

As seen in the definition by extension, $p^\prime_{S,j}$ is defined by cases,
one for each possible pair $i, \ell$ of nodes and colors such that the node~$i$ is
the least central element of the $P_\ell$ color class, where~$P_\ell$
is the first star-shaped
color class. The extended Frege proof then shows
that $\neg \kneser n k(\vecp)$ implies $\neg \kneser{n-1}k(\vecpprime)$,
i.e., that if the variables $p_{S,j}$ define a coloring, then the
variables $p^\prime_{S,j}$
also define a coloring.  For this, it is necessary to show that
there is at least one star-shaped color class;
this is provable with a polynomial size extended Frege proof
(even a Frege proof) using the construction of Lemma~\ref{lem:SomeStarShaped}
and the counting techniques of~\cite{Buss:PHP}.

The extended Frege proof iterates this process of removing one node and one color
until it is shown that there is a  coloring of
$\binom{N(k)}{k}$. This is then refuted by exhaustively considering all graphs
with $\le N(k)$ nodes.
\qed

\subsection{Quasi-polynomial Size Frege Proofs}
\label{sec:FormalF}

This section discusses some of the details of the formalization of the
argument in Sect.~\ref{sec:KneserLovaszFrege} as quasi-polynomial size
Frege proofs, establishing
Theorem~\ref{thm:MainThmFrege}.
 First we will form an extended Frege proof, then modify it to become
a Frege proof.  As before, the proof starts with the assumption that
$\kneser n k(\vecp)$ is false.  As we describe next,
the extended Frege proof then introduces variables $\vecpprime$
by extension so that $\kneser{n-n/2k}{k}$ is false.  This process will
be repeated $O(\log n)$ times.  The final Frege proof is obtained by
unwinding the definitions by extension.

For a set~$X$ of formulas and $t>0$, let ``$|X| < t$"
denote a formula that is true when the number of true formulas
in~$X$ is less than~$t$. ``$|X|<t$'' can be
expressed by a formula of size polynomially bounded by
the total size of the formulas in~$X$, using the
construction in~\cite{Buss:PHP}.
``$|X|=t$" is defined similarly.

The formulas $\starshaped(i, \ell)$ and $\starshaped(\ell)$
are the same as in Sect.~\ref{sec:FormalEF}. A color~$\ell$
is now discarded if it is among the least $n/2k$ star-shaped
color classes.
\[
\discardcolor(\ell) ~:=~
   \starshaped(\ell) \wedge
   \left ( |\{\starshaped(\ell^\prime) : \ \ell^\prime \le \ell    \} |
   \le  n/2k    \right )
\]
The discarded nodes are the least central elements of the
discarded color classes.
\[
\discardnode(i) ~:=~ \bigor_{\ell} \Bigl[ \discardcolor(\ell) \land
    \starshaped(i,\ell) \land \bigand_{i^\prime<i} \neg\starshaped(i^\prime,\ell) \Bigr].
\]
The remaining, non-discarded colors and nodes are
renumbered to form an instance of $\kneser{n-n/2k}{k}$.  For this,
the formula $\renumnode(i^\prime,i)$ is true when the node~$i^\prime$ is the $i$th node that is
not discarded; similarly $\renumcolor(j^\prime,j)$ is true
when the color~$j^\prime$ is the $j$th color that
is not discarded.
\begin{eqnarray*}
\renumnode(i^\prime,i) &\kern-3pt:=&
\left ( | \{ \neg \discardnode (i^{\prime\prime})
     \kern-1pt:\kern-1pt i^{\prime\prime}{<}i^\prime \}| \kern-1pt=\kern-1pt i{-}1  \right ) \wedge \neg \discardnode(i^\prime)
\\[0.5ex]
\renumcolor(j^\prime,j) &\kern-3pt:=&
\left ( | \{ \neg \discardcolor (j^{\prime\prime})
     \kern-1pt:\kern-1pt j^{\prime\prime}{<}j^\prime \} | \kern-1pt=\kern-1pt j{-}1  \right ) \wedge \neg \discardcolor(j^\prime)
\end{eqnarray*}

For each $S=\{i_1, \dots, i_k \} \in \binom{n-n/2k}{k}$ and
$j \in \bracket{(n-n/2k)-2k+1}$, we define by extension
\[
p^\prime_{S,j} \equiv \bigor_{i^\prime_1, \dots i^\prime_k, j^\prime } \left ( \bigand_{t=1}^k \left(\renumnode(i^\prime_t,i_t) \right) \wedge \renumcolor(j^\prime,j) \wedge p_{\{i^\prime_1, \dots, i^\prime_k\},j^\prime}     \right ).
\]

The Frege proof then
argues that if the variables $p_{S,j}$
define a coloring, then the variables $p^\prime_{S,j}$ define a coloring,
i.e., that
$\neg \kneser{n}{k}(\vecp) \rightarrow \neg \kneser{n-n/2k}{k}(\vecpprime)$.
The main step for this is proving there are at least $n/2k$ star-shaped color
classes by formalizing the proof of Lemma~\ref{lem:ManyStarShaped}; this can be done
with polynomial size Frege proofs using the counting techniques
from~\cite{Buss:PHP}.
After that, it is straightforward to prove that, for each $S\in\binom{n-n/2k}{k}$ and
$j\in\bracket{(n-n/2k)-2k+1}$, the variable~$p_{S,j}^\prime$ is well-defined;
and that the $\vecpprime$ collectively falsify $\kneser{n-n/2k}{k}$.

This is iterated $O(\log n)$ times until fewer than
$N(k,1/2)$ nodes remain.  The proof concludes
with a hard-coded proof that there are no such colorings of
the finitely many small Kneser graphs.

To form the quasi-polynomial size Frege proof, we unwind the
definitions by extension. Each definition by extension was polynomial size;
they are nested to a depth of $O(\log n)$.  So the resulting Frege proof
is quasi-polynomial size.
\qed

\section{The Truncated Tucker Lemma}\label{sec:Tucker}

This section introduces the truncated Tucker lemma.
The usual (octahedral) Tucker lemma implies the truncated Tucker lemma
and the truncated Tucker lemma implies the Kneser-Lov\'asz theorem.
The truncated Tucker lemma is of particular interest, since its
propositional translations are only polynomial size; in contrast,
the propositional translations of the usual Tucker lemma
are of exponential size.
Additionally, there are polynomial size constant depth Frege proofs
of the Kneser-Lov\'asz tautologies from the truncated Tucker tautologies.

Our definition and proof of the truncated Tucker lemma
borrows techniques and notation from
Matou{\v s}ek~\cite{Matousek:ComboProofsKneser}.

\begin{definition}
Let $n\ge 1$. The {\em octahedral ball~$\mathcal{B}^n$} is:
\[
\mathcal{B}^n ~:=~
    \{(A,B) : A,B \subseteq \bracket{n} \text{ and } A \cap B = \emptyset \}.
\]
Let $1 \le k \le n$. The {\em truncated octahedral ball~$\mathcal{B}^n_k$} is:
\[
\mathcal{B}^n_k ~:=~
  \Bigl \{(A,B) :  A,B \in \binom n k \cup \{\emptyset\},\,
A \cap B = \emptyset, \text{ and }
 (A,B) \neq (\emptyset,\emptyset) \Bigr \}.
\]
\end{definition}

\begin{definition}
Let $n>1$. A mapping
$\lambda: \mathcal{B}^n \to \{1, \pm 2, \dots, \pm n\}$
is {\em antipodal}
if $\lambda(\emptyset,\emptyset) = 1$,
and for all other pairs $(A,B)\in \mathcal{B}^n$,
$\lambda(A,B) = -\lambda(B,A)$.

Let $n\ge 2k>1$. A mapping
$\lambda: \mathcal{B}^n_k \to \{\pm 2k, \dots, \pm n\}$
is {\em antipodal} if for all
$(A,B) \in \mathcal{B}^n_k$,
$\lambda(A,B) = -\lambda(B,A)$.

\end{definition}
Note that $-1$ is not in the range of~$\lambda$, and
$(\emptyset, \emptyset)$ is the only member of~$\mathcal{B}^n$
that is
mapped to~$1$ by~$\lambda$.

For $A \subseteq \bracket{n}$, let $A_{\le k}$ denote the least~$k$
elements of~$A$.  By convention $\emptyset_{\le k} = \emptyset$,
but otherwise the notation is used only when $|A|\ge k$.

The Tucker lemma uses the subset relation~$\subseteq$ on~$\bracket n$, but the
truncated Tucker lemma uses instead a stronger partial order~$\preceq$
on~$\binom n k$.

\begin{definition}
Let $\preceq$ be the partial order on sets in
$\binom{n}{k} \cup \{\emptyset\}$ defined by $A_1 \preceq A_2$ iff
$(A_1 \cup A_2)_{\le k} = A_2$.
\end{definition}

\begin{lemma}\label{lem:precpartialorder}
The relation $\preceq$ is a partial order with $\emptyset$ its least element.
\end{lemma}

\begin{proof}
It is clearly reflexive. For anti-symmetry, $A_1 \preceq A_2$ and $A_2 \preceq A_1$
imply that $A_1 = (A_1 \cup A_2)_{\le k} = (A_2 \cup A_1)_{\le k} = A_2$.
For transitivity: Suppose $A_1 \preceq A_2$ and $A_2 \preceq A_3$.
Then $(A_1 \cup A_2)_{\le k}=A_2$ and $(A_2 \cup A_3)_{\le k}=A_3$.
This implies that
\[
A_3 = (A_2 \cup A_3)_{\le k} = ((A_1 \cup A_2)_{\le k} \cup A_3)_{\le k}=(A_1 \cup (A_2 \cup A_3)_{\le k})_{\le k} = (A_1 \cup A_3)_{\le k}.
\]
Therefore $A_1 \preceq A_3$. That $\emptyset$ is the least element is clear
from the definition.
\qed
\end{proof}

\begin{definition}\label{def:complementary}
Two pairs $(A_1, B_1)$ and $(A_2, B_2)$ in $\mathcal{B}^n$ are
{\em complementary} w.r.t.\ an antipodal map~$\lambda$
on~$\mathcal{B}^n$ if
$A_1 \subseteq A_2$, $B_1 \subseteq B_2$ and
$\lambda(A_1, B_1) = - \lambda(A_2, B_2)$.

For $(A_1, B_1)$ and $(A_2, B_2)$ in $\mathcal{B}^n_k$,
write $(A_1, B_1) \preceq (A_2, B_2)$ when
$A_1 \preceq\penalty10000 A_2$, $B_1 \preceq B_2$,
and $A_i \cap B_j = \emptyset$ for $i,j \in \{1,2\}$.
The pairs $(A_1, B_1)$ and $(A_2, B_2)$ are
{\em $k$-complementary} w.r.t.\ an antipodal map~$\lambda$
on~$\mathcal{B}^n_k$
if $(A_1, B_1) \preceq (A_2, B_2)$ and
$\lambda(A_1, B_1) = - \lambda(A_2, B_2)$.
\end{definition}

\begin{theorem}[Tucker lemma]\label{thm:Tucker}
If $\lambda: \mathcal{B}^n \to \{1, \pm 2, \dots, \pm n \}$ is
antipodal, then there are two elements in~$\mathcal{B}^n$ that are
complementary.
\end{theorem}

\begin{theorem}[Truncated Tucker]\label{thm:TruncatedTucker}
Let $n \ge 2k>1$. If $\lambda: \mathcal{B}^n_k \to \{\pm 2k \dots, \pm n \}$ is
antipodal, then there are two elements in~$\mathcal{B}^n_k$ that
are $k$-complementary.
\end{theorem}

For a proof of Theorem~\ref{thm:Tucker}, see \cite{Matousek:ComboProofsKneser}.
An appendix to the arXiv version of this paper proves Theorem~\ref{thm:TruncatedTucker}
from Theorem~\ref{thm:Tucker}.

The truncated Tucker lemma has polynomial size propositional translations.
For each $(A,B) \in \mathcal{B}^n_k$, and for each $i \in \{\pm 2k, \dots, \pm n\}$,
let $p_{A,B,i}$ be a propositional variable with the intended meaning that
$p_{A,B,i}$ is true when $\lambda(A,B) = i$. The following
formula $\text{Ant}(\vecp)$ states that the map is total and antipodal:
\[
\bigand_{(A,B) \in \mathcal{B}^n_k}
 \,\, \bigor_{i \in \{\pm 2k, \dots, \pm n\}}
  ( p_{A,B,i} \land p_{B,A,-i} ).
\]
The following formula $\text{Comp}(\vecp)$ states that there
exists two elements in~$\mathcal{B}^n_k$ that are $k$-complementary:
\[
\bigor_{\substack{  (A_1,B_1) ,(A_2,B_2) \in \mathcal{B}^n_k,\\
(A_1,B_1) \preceq (A_2,B_2) \\  i \in \{\pm 2k, \dots, \pm n\} } }  \left ( p_{A_1,B_1,i} \wedge p_{A_2,B_2,-i} \right ).
\]
The truncated Tucker tautologies are defined to be
$\text{Ant}(\vecp) \limplies \text{Comp}(\vecp)$.
(We could add an additional hypothesis, that for each $A,B$ there is
at most one~$i$ such that $p_{A,B,i}$, but this is not needed for
the Tucker tautologies to be valid.)
There are $< n^{2k}$ members $(A,B)$ in~$\mathcal{B}^n_k$.
Hence, for fixed~$k$, there are only polynomially many
variables $p_{A,B,i}$, and the truncated Tucker
tautologies have size polynomially bounded by~$n$.
On the other hand, the propositional translation of the
usual Tucker lemma requires an exponential number of propositional variables
in~$n$, since the cardinality of~$\mathcal{B}^n$ is exponential in~$n$.

\begin{proof}[Theorem~\ref{thm:KneserLovasz} from the truncated Tucker lemma]
Let $c: \binom{n}{k} \to \{2k, \dots, n\}$ be a $(n{-}2k{+}1)$-coloring of
$\binom n k$. We show that this implies the existence of an antipodal
map $\lambda$ on $\mathcal{B}^n_k$ that has no $k$-complementary
pairs.
Let $\le$ be a total order on $\binom{n}{k} \cup \{\emptyset\}$ that refines the
partial order $\preceq$. Define $\lambda(A,B)$ to be $c(A)$ if
$A > B$ and $-c(B)$ if $B > A$.
We argue that there are no $k$-complementary pairs in~$\mathcal{B}^n_k$
with respect to $\lambda$. Suppose there are, say $(A_1, B_1)$ and $(A_2, B_2)$.
Since $\lambda$ must assign these opposite signs, either
$A_1 < B_1 \le B_2 < A_2$ or $B_1 < A_1 \le A_2 < B_2$.
In the former case it must be that, $c(B_1) = c(A_2)$ and in the latter case
that $c(A_1) = c(B_2)$. Since $B_1 \cap A_2$ and $A_1 \cap B_2$ are empty
in either case we have a contradiction, since $c$ was assumed to be a
coloring.
\qed
\end{proof}

The above proof of the Kneser-Lov\'asz theorem from the truncated Tucker lemma
can be readily translated into polynomial size constant depth Frege proofs.

\begin{question}
Do the propositional translations of the Truncated Tucker lemma
have short (extended) Frege proofs?
\end{question}

\bibliographystyle{splncs03}
\bibliography{logic,cstheory}

\long\def\eat#1{\relax}

\section{Appendix: Proof of the Truncated Tucker Lemma}
\label{sec:TruncatedTuckerProof}

We prove the truncated Tucker lemma from the Tucker lemma:

\begin{proof}[of Theorem~\ref{thm:TruncatedTucker} from Theorem~\ref{thm:Tucker}]
We show the contrapositive.
Suppose Theorem~\ref{thm:TruncatedTucker} is false.
In other words, there is some antipodal $\lambda: \mathcal{B}^n_k \to \{\pm 2k, \dots, \pm n\}$
and there are no $k$-complementary pairs of elements
in~$\mathcal{B}^n_k$ with respect to~$\lambda$.
We will define an antipodal map
$\lambda^\prime: \mathcal{B}^n \to \{1, \pm 2, \dots, \pm n\}$
that has no complementary pairs of elements in~$\mathcal{B}^n$ with respect
to~$\lambda^\prime$, in violation of Theorem~\ref{thm:Tucker}.

Let $\le$ be a total order on the subsets
of~$\bracket{n}$ that respects cardinalities
and refines $\preceq$ on elements in $\binom{n}{k} \cup \{\emptyset \}$.
In other words, if $|A| < |B|$  or if $A \preceq B$
then $A \le B$.
Similarly to a construction
in~\cite{Matousek:ComboProofsKneser},
$\lambda^\prime(A,B)$ is defined by cases:

\begin{enumerate}[C{a}se I:]
\item If $|A|<k$ and $|B|<k$, then
\[
\lambda^\prime(A,B) = \begin{cases} 1+|A|+|B| & \text{if } A \ge B \\
-(1+|A|+|B|)                                  & \text{if } B > A.
\end{cases}
\]

\item If $|A|=k$ and $|B|<k$, then $\lambda^\prime(A,B) = \lambda(A,\emptyset)$.
Similarly, if $|A|<k$ and $|B|=k$, then $\lambda^\prime(A,B) = \lambda(\emptyset,B)$.

\item If $|A| \ge k$ and $|B| \ge k$,
then $\lambda^\prime(A,B) = \lambda(A_{\le k},B_{\le k})$.
\end{enumerate}

The map $\lambda^\prime$ is clearly antipodal since $\lambda$ is.
Let $(A_1, B_1)$ and $(A_2, B_2)$ be members
of~$\mathcal{B}^n$ with $A_1 \subseteq A_2$, $B_1 \subseteq B_2$.
We wish to show they are not a complementary pair for~$\lambda^\prime$,
namely that
$\lambda^\prime(A_1,B_1) \neq - \lambda^\prime(A_2, B_2)$.
The argument splits into cases.

\begin{enumerate}
\item $\lambda^\prime(A_1,B_1)$ and $\lambda^\prime(A_2, B_2)$
are assigned by Case~I.
If $A_1=A_2$ and $B_1=B_2$, then clearly
$\lambda^\prime(A_1,B_1) \neq - \lambda^\prime(A_2,B_2)$.
Otherwise, at least one of the inclusions
$A_1 \subseteq A_2$ and $B_1 \subseteq B_2$ is proper, so
$|A_1|+|B_1| < |A_2|+|B_2|$.
But $|\lambda^\prime(A_1,B_1)|=1+|A_1|+|B_1|$
and $|\lambda^\prime(A_2,B_2)|=1+|A_2|+|B_2|$, so
$\lambda^\prime(A_1,B_1) \neq - \lambda^\prime(A_2,B_2)$.

\item $\lambda^\prime(A_1, B_1)$ and $\lambda^\prime(A_2, B_2)$
are assigned by Case~II.
Without loss of generality, $|A_1|=k$ and $|B_1|<k$.
So then $A_2=A_1$ and $|B_2|<k$. But then
$\lambda^\prime(A_1,B_1) = \lambda(A_1,\emptyset)$ and
$\lambda^\prime(A_2, B_2)=\lambda(A_1, \emptyset)$,
so
$\lambda^\prime(A_1,B_1) \neq - \lambda^\prime(A_2,B_2)$.

\item $\lambda^\prime(A_1, B_1)$ and $\lambda^\prime(A_2,B_2)$
are both assigned by Case~III.
It is clear that $A_{i,\le k} \cap B_{j, \le k} = \emptyset$ for $i, j \in \{1,2\}$,
since $A_2 \cap B_2 = \emptyset$.
We claim that $A_{1, \le k} \preceq A_{2,\le k}$ and
$B_{1,\le k} \preceq B_{2,\le k}$.
This is because $A_1 \subseteq A_2$, so
\[
(A_{1,\le k} \cup A_{2,\le k})_{\le k} = (A_1 \cup A_2)_{\le k} = A_{2,\le k}.
\]
Therefore $A_{1, \le k } \preceq A_{2, \le k}$.
The same argument shows that $B_{1, \le k } \preceq B_{2, \le k}$.
Since there are no $k$-complementary pairs with respect to~$\lambda$,
it must be that $\lambda(A_{1,\le k},B_{1, \le k}) \neq - \lambda(A_{2,\le k},B_{2,\le k})$.
Also, $\lambda^\prime(A_1,B_1) = \lambda(A_{1, \le k}, B_{1, \le k})$
and $\lambda^\prime(A_2, B_2) = \lambda(A_{2, \le k},B_{2, \le k})$.
Hence $\lambda^\prime(A_1, B_1) \neq - \lambda^\prime(A_2, B_2)$.

\item $\lambda^\prime(A_1, B_1)$, and $\lambda^\prime(A_2, B_2)$
are assigned by Cases II and~III.
If this happens, it must be that $(A_1, B_1)$ is assigned
as in Case II, and $(A_2, B_2)$ is assigned as in Case III. Without loss of
generality, say $|A_1|=k$, and $|B_1|<k$.
We show that $(A_1, \emptyset) \preceq (A_2, B_{2,\le k})$.
By the same argument as the previous case,
$A_{1, \le k} \preceq A_{2,\le k}$.
As remarked earlier, $\emptyset \preceq B_{2, \le k}$, since
the empty set is the least element in the $\preceq$ partial order.
Also, the four sets $A_{1,\le k} \cap \emptyset$,
$A_{1, \le k} \cap B_{2, \le k}$,
$\emptyset \cap B_{2, \le k}$, and $A_{2, \le k} \cap B_{2, \le k}$
are empty.

Since there are no $k$-complementary pairs with respect to~$\lambda$,
it must be that $\lambda(A_{1,\le k},\emptyset) \neq - \lambda(A_{2,\le k},B_{2,\le k})$.
By definition, $\lambda^\prime(A_1,B_1) = \lambda(A_{1, \le k}, \emptyset)$
and $\lambda^\prime(A_2, B_2) = \lambda(A_{2, \le k},B_{2, \le k})$.
Hence $\lambda^\prime(A_1, B_1) \neq - \lambda^\prime(A_2, B_2)$.

\item The only remaining case is when one of $\lambda^\prime(A_1, B_1)$
and $\lambda^\prime(A_2, B_2)$ is assigned by Case~I, and the other is assigned
by Case II or~III. It must be that $\lambda^\prime(A_1, B_1)$ is assigned
by Case~I and $\lambda^\prime(A_2, B_2)$ by Case II or~III.
Observe that $|\lambda^\prime(A_1, B_1)|=1+|A_1|+|B_1| \le 2k-1$,
and that $|\lambda^\prime(A_2,B_2)| \ge 2k$. Therefore
$\lambda^\prime(A_1,B_1) \neq -\lambda^\prime(A_2,B_2)$.	
\end{enumerate}

This establishes that $\lambda^\prime$ is an antipodal map
with no complementary pairs, hence Theorem~\ref{thm:Tucker} is false.
This completes the proof of the contrapositive.
\qed
\end{proof}

\section{Appendix: Optimal Colorings of Kneser Graphs}\label{sec:OptimalColoring}

It is well-known that $\binom{n}{k}$ has
an $(n-2k+2)$-coloring~\cite{Lovasz:KneserTheorem}.
A simple construction of such a coloring, which
we call~$c_1$, is given here for completeness as follows.
For $S \in \binom{n}{k}$, define $c_1(S)$ by:
\begin{itemize}
\item[\rm(1)]  If $S \not\subseteq \bracket{2k-1}$,
let $c_1(S) = \max(S)-(2k-2)$.  Clearly $1 < c_1(S) \le n-2k+2$.
\item[\rm(2)] Otherwise, let $c_1(S) = 1$.
\end{itemize}
We claim that $c_1$ defines a proper coloring.
By construction, if $c_1(S)>1$, then $c_1(S)+(2k-2)\in S$.
Thus, if $c_1(S) = c_1(S^\prime) > 1$, then $S\cap S^\prime \not=\emptyset$
and $S$ and~$S^\prime$ are not joined by an edge in the Kneser graph.
On the other hand, if $c_1(S)=1$,
then $S$ contains $k$ elements from
the set $[2k-1]$.
Any two such
subsets have nonempty intersection, and therefore
if $c_1(S) = c_1(S^\prime) = 1$,
then again $S \cap S^\prime \not= \emptyset$.
Note that $c_1$ contains $n-2k+1$ many star-shaped
color classes, and only one non-star-shaped color class.

In view of Lemma~\ref{lem:ManyStarShaped}, it is interesting
to ask whether it is possible to give $(n-2k+2)$-colorings
with fewer star-shaped color classes and more non-star-shaped
color classes.  The next theorem
gives the best construction we know.
\begin{theorem}\label{thm:fewstarshapes}
Let $k\ge 1$ and $n\ge 3k+3$.  There is an
$(n{-}2k{+}2)$ coloring~$c_{k-1}$
of $\binom n k$
which has $k-1$ many non-star-shaped color classes and only
$n-3k+3$ many star-shaped color classes.
\end{theorem}

\begin{proof}
To construct $c_{k-1}$, partition the set
$\bracket n$ into $n-2k+2$ many subsets
$T_1, \ldots, T_{n-2k+2}$ as follows.  For $i\le n-3k+3$,
$T_i$~is chosen to be a singleton set, say
$T_i=\{n-i+1\}$.  The remaining $k-1$ many
$T_i$'s are subsets of size~3, say
$T_i = \{ j-2, j-1, j \}$ where $j=i-(n-3k+3)$.
Since $n = (n-3k+3) + 3(k-1)$,
the sets $T_i$ partition~$\bracket n$, and
each $T_i$ has cardinality either 1 or~3.
For $S$ a subset of~$n$ of cardinality~$k$, define
the color $c_{k-1}(S)$ to equal the least $i$ such that
\[
| S\cap T_i | ~>~ \frac 12 |T_i|.
\]
We claim there must exist such an~$i$.  If not,
then $S$ contains no members of the singleton subsets~$T_i$
and at most one member of each of the subsets~$T_i$
of size three.  But there are only $k-1$ many subsets of size
three, contradicting $|S|=k$.

It is easy to check that
if $c_{k-1}(S)=c_{k-1}(S^\prime)$
then $S\cap S^\prime \not= \emptyset$.
Thus $c_{k-1}$ is a coloring.
Furthermore, $c_{k-1}$ has $k-1$ many non-star-shaped classes
and $n-3k+3$ many star-shaped classes.
\qed
\end{proof}

Theorem~\ref{thm:fewstarshapes} can be extended
to show that when $n<3k+3$, there is a $n-2k+2$ coloring
with no star-shaped color class.
The proof construction uses a similar idea, based on the
fact that $\bracket n$ can be partitioned into
$n-2k+2$ many subsets, each of odd cardinality~$\ge 3$.
We leave the details to the reader.

\begin{question}
Do there exist $(n-2k+2)$-colorings of the $(n,k)$-Kneser graphs
with more than $k-1$ many non-star-shaped color classes?
\end{question}

\end{document}